\documentclass{article}
\usepackage{amsmath,amsxtra,amssymb,amsthm,amsfonts}
\usepackage{graphicx} 
\usepackage{color,graphicx}
\usepackage[margin=1.05in]{geometry}
\usepackage[dvipsnames]{xcolor}
\usepackage{url}

\usepackage[pdftex,linktocpage=true,colorlinks,citecolor=blue,linkcolor=blue,pagebackref]{hyperref}
\usepackage{hyperref}
\usepackage{cleveref}
\usepackage[alphabetic,backrefs]{amsrefs}
\usepackage{caption}
\usepackage{subcaption}

\usepackage{pgfplots}
\usepackage{pgf}
\usepackage{tikz}
\usetikzlibrary{patterns}
\usetikzlibrary{arrows.meta}
\usepgfplotslibrary{patchplots} 
\usetikzlibrary{pgfplots.patchplots} 
\pgfplotsset{width=9cm,compat=1.5.1}

\newcommand\ip[2]{\ensuremath{\langle#1 , #2\rangle}}

\newcommand\tr{\mathop{\rm tr}\nolimits}

\newcommand{\norm}[1]{\|#1\|}

\def\C{\mathbb{C}}

\def\R{\mathbb{R}}
\def\F{\mathbb{F}}

\def\M{\mathcal{M}}
\def\S{\mathcal{S}}

\def\a{\mathbf{a}}
\def\b{\mathbf{b}}

\def\e{\mathbf{e}}

\def\v{\mathbf{v}}
\def\w{\mathbf{w}}
\def\x{\mathbf{x}}
\def\y{\mathbf{y}}
\def\z{\mathbf{z}}
\def\0{\mathbf{0}}

\numberwithin{equation}{section}
\newtheorem{lemma}{Lemma}[section]

\newtheorem{theorem}[lemma]{Theorem}
\newtheorem{corollary}[lemma]{Corollary}
\newtheorem{conjecture}[lemma]{Conjecture}

\newtheorem{fact}{Fact} 
\newtheorem{example}{Example}

\begin{document}

\title{
    Generalizing the Cauchy--Schwarz inequality:\\
    Hadamard powers and tensor products
}

\author{  
    Nathaniel Johnston\textsuperscript{1}, Sarah Plosker\textsuperscript{2}, Charles Torrance\textsuperscript{1}, and Luis M.~B.~Varona\textsuperscript{1,3}
}

\maketitle

\begin{abstract}
    We explore and generalize a Cauchy--Schwarz-type inequality originally proved in [Electronic Journal of Linear Algebra 35, 156--180 (2019)]:\[\|\v^2\|\|\w^2\| - \ip{\v^2}{\w^2} \leq \|\v\|^2\|\w\|^2 - \ip{\v}{\w}^2 \quad \text{for all} \quad \v,\w \in \R^n.\]We present three new proofs of this inequality that better illustrate ``why'' it is true and generalize it in several different ways: we generalize from vectors to matrices, we explore which exponents other than $2$ result in the inequality holding, and we derive a version of the inequality involving three or more vectors.\\
    
    \medskip

    \noindent \textbf{Keywords:}  Cauchy--Schwarz inequality, matrix inequalities, linear algebra\medskip
	
	\noindent \textbf{MSC2010 Classification:}  
  97H60; 
  15A18; 
  15A45 
\end{abstract}

\addtocounter{footnote}{1}
\footnotetext{Department of Mathematics \& Computer Science, Mount Allison University, Sackville, NB, Canada E4L 1E4}
\addtocounter{footnote}{1}
\footnotetext{Department of Mathematics \& Computer Science, Brandon University, Brandon, MB, Canada R7A 6A9}
\addtocounter{footnote}{1}
\footnotetext{Department of Politics \& International Relations, Mount Allison University, Sackville, NB, Canada E4L 1E4}

\section{Introduction}\label{sec:intro}

The Cauchy--Schwarz inequality states that $|\langle \v ,\w \rangle |\leq \|\v \|\|\w \|$ for all vectors $\v, \w\in \R^n$, with equality precisely when $\v$ and $\w$ are linearly dependent. In \cite{MO_Lemma,JM19}, the following inequality involving vectors $\v,\w \in \R^n$ was proved:
\begin{align}\label{eq:cs_original}
    \|\v^2\|\|\w^2\| - \ip{\v^2}{\w^2} \leq \|\v\|^2\|\w\|^2 - \ip{\v}{\w}^2,
\end{align}
where $\v^2, \w^2 \in \R^n$ are the vectors obtained by squaring each entry in $\v$ and $\w$, respectively, and $\ip{\v}{\w}$ is the usual inner (dot) product on $\R^n$. Inequality~\eqref{eq:cs_original} may be thought of as a strengthening of the Cauchy--Schwarz inequality: while the latter simply says that both sides are non-negative, the former also provides a nontrivial lower bound on the right-hand side.

Inequality~\eqref{eq:cs_original} was used by \cite{JM19} when studying \emph{pairwise completely positive} (PCP) matrices: pairs of matrices that share a joint decomposition so
that one of them is necessarily positive semidefinite while the other one is necessarily entrywise non-negative. Inequality~\eqref{eq:cs_original} is also of independent interest in the general study of inequalities, where various generalizations of the Cauchy--Schwarz inequality have been developed. One of the most famous generalizations is H{\"o}lder's inequality for integrals of measurable functions in L$^p$-spaces. Sums of products of powers of entries of the given vectors were considered in \cite{callebaut1965generalization}. Generalized Cauchy--Schwarz inequalities for unitarily invariant norms of operators were provided in \cite{bhatia1995cauchy}. Cauchy--Schwarz-like inequalities and historical notes are presented in \cite{labropoulou2024generalizations}. Entire textbooks are devoted to mathematical inequalities, including  \cite{steele2004cauchy}, which uses the Cauchy--Schwarz inequality as a guide through the study of the most fundamental mathematical inequalities. This is by no means an exhaustive list of work done to generalize or strengthen the Cauchy--Schwarz inequality, but it serves as useful context nonetheless.

While Inequality~\eqref{eq:cs_original} has indeed been proven, it remains somewhat mysterious mathematically---in particular, the presence of the vectors $\v^2$ and $\w^2$ seems, at first glance, rather unnatural. The goal of this paper is to provide alternative proofs of Inequality~\eqref{eq:cs_original} and subsequently generalize it in several different ways, thus shedding some light on ``why'' it is true. A more specific list of the main goals and results of this paper follows:

\begin{itemize}
    \item We present a geometric interpretation of Inequality~\eqref{eq:cs_original} and discuss how it can be proved by interpreting the quantities on both sides of the inequality as areas (see Section~\ref{sec:intuitive_geometric}).
    
    \item We present several results concerning what happens if we change the exponent $2$ in Inequality~\eqref{eq:cs_original} to another value $p$. In particular, we show that the inequality is false when $p \in (1,2)$ (see Example~\ref{exam:p12}) and is true when $p \geq 1$ is an integer (see Corollaries~\ref{cor:equal_tensors} and~\ref{cor:equal_tensors_from_lagrange}). Additionally, we conjecture that the inequality continues to hold even when $p \geq 2$ is not an integer (see Conjecture~\ref{conj:non_integer}).

    \item We present several other generalizations of Inequality~\eqref{eq:cs_original} to more than two vectors (see Theorem~\ref{thm:generalize_many_vecs}, Corollary~\ref{cor:fp_projection_diagonal}, and Corollary~\ref{cor:tripartite_cs}). Along the way, we present multiple results that we believe to be of independent interest, including a relationship between the vectors of diagonal entries and eigenvalues of a pair of matrices (see Corollary~\ref{cor:gen_to_matrices_eig}) and a sum-of-squares decomposition generalizing Lagrange's identity (see Theorem~\ref{thm:sumofsquares}).
\end{itemize}

We note that while there is a long history of trying to generalize and strengthen the Cauchy--Schwarz inequality, none of these existing generalizations lead to Inequality~\eqref{eq:cs_original} or any of the other results included in the present work.

\subsection{Notation and Mathematical Preliminaries}\label{sec:notation}

We use $\M_n(\F)$ to denote the set of $n \times n$ matrices with entries from the field $\F \in \{\R,\C\}$. The trace norm of a matrix $A \in \M_n(\F)$ is the sum of the singular values of $A$---that is, $\|A\|_{\tr} = \tr(\sqrt{A^*A})$, where $A^*$ denotes the conjugate transpose of the matrix $A$ (this is simply the usual transpose when $\F = \R$). The Frobenius norm of a matrix $A\in \M_n(\F)$ is defined by $\|A\|_{\textup{F}}=\sqrt{\sum_{i,j=1}^n|a_{i,j}|^2}=\sqrt{\tr(AA^*)}$, which also equals the $2$-norm of the vector of singular values of $A$. The Frobenius inner product of two matrices $A, B\in \M_n(\F)$ is similarly defined by ${\ip{A}{B}}_{\textup{F}} = \tr(A^*B)=\sum_{i,j}{\overline {a_{i,j}}}b_{i,j}$. 

Given a matrix $A\in \M_n(\F)$, its vectorization $\mathrm{vec}(A)$ is the $n^2$-dimensional vector whose entries are those of the first column of $A$, then of the second column of $A$, and so on:
\[
    \mathrm{vec}(A) = (a_{1,1}, \ldots, a_{n,1}, a_{1,2}, \ldots, a_{n,2}, \ldots, a_{1,n}, \ldots, a_{n,n}).
\]
The tensor (Kronecker) product of two matrices $A\in \M_m(\F)$ and $B\in \M_n(\F)$ is defined by
\[
    A\otimes B = \begin{bmatrix}a_{1,1}{B} &\cdots &a_{1,n}{B} \\\vdots &\ddots &\vdots \\a_{n,1}{B} &\cdots &a_{n,n}B \end{bmatrix}
\]
and the tensor (Kronecker) product of vectors $\x\in \F^m$ and $\y\in \F^n$ is defined similarly by
\begin{eqnarray*}
\x\otimes \y&=&(x_{1}\y, x_{2}\y, \ldots, x_m\y)\\
&=&(x_1y_1, \ldots, x_1y_n, x_2y_1, \ldots, x_2y_n, \ldots, x_my_1, \ldots, x_my_n).
\end{eqnarray*}
Whenever the shape of a vector is important (e.g., if a matrix multiplies it), we treat it as a column. The standard basis of $\R^n$ and $\C^n$ is $\{\mathbf{e}_1, \ldots, \mathbf{e}_n\}$, where $\mathbf{e}_i$ is the $n$-dimensional vector with $1$ in the $i$-th component and $0$ everywhere else.

\section{Intuitive/Geometric Proof}\label{sec:intuitive_geometric}

We now present an (at least partially) geometric argument for why Inequality~\eqref{eq:cs_original} holds that we believe is more enlightening than the proof originally given in \cite{JM19}. To this end, consider the pair of functions $f,g : \R^n \times \R^n \rightarrow \R$ defined by
\begin{align}\label{eq:cs_func}
    f(\v,\w) & = \|\v\|\|\w\| - \ip{\v}{\w} \quad \text{and} \quad g(\v,\w) = \sqrt{\|\v\|^2\|\w\|^2 - \ip{\v}{\w}^2}.
\end{align}
It is well known that the function $g$ is exactly the area of the parallelogram with sides $\v$ and $\w$. This is because $g$ is the Grammian of $\v$ and $\w$---that is, if $A$ is the $n \times 2$ matrix with $\v$ and $\w$ as its columns, then $g(\v,\w) = \sqrt{\det(A^*A)}$.

To provide a similar geometric interpretation of $f$, first notice that both $f$ and $g$ are unitarily invariant (i.e., $f(U\v,U\w) = f(\v,\w)$ and $g(U\v,U\w) = g(\v,\w)$ for all unitary matrices $U \in \M_n(\R)$), so they can be written as functions of $\|\v\|$, $\|\w\|$, and the angle $\theta$ between $\v$ and $\w$. In particular, the fact that $\cos(\theta) = \ip{\v}{\w}/(\|\v\|\|\w\|)$ quickly implies the following formulas:
\begin{align}\label{eq:cs_func_angle}
    f(\v,\w) & = \|\v\|\|\w\|(1 - \cos(\theta)) \quad \text{and} \quad g(\v,\w) = \|\v\|\|\w\|\sin(\theta).
\end{align}

Using the fact that $1-\cos(\theta) = 2\sin^2(\theta/2)$, we see that
\begin{align}\label{eq:f_g_relationship}
    f(\v,\w) = 2(g(\x,\y))^2,
\end{align}
where $\x$ and $\y$ are any two vectors with lengths $\sqrt{\|\v\|}$ and $\sqrt{\|\w\|}$, respectively, with an angle of $\theta/2$ between them (see Figure~\ref{fig:geometric_proof}). Since $g$ measures the area of a parallelogram and orthogonal projections cannot increase areas, it follows immediately that $g(P\v,P\w) \leq g(\v,\w)$ whenever $P \in \M_n(\R)$ is an orthogonal projection. By making use of Equation~\eqref{eq:f_g_relationship}, we find that the same is true of the function $f$:\footnote{We will provide an algebraic proof of this fact a bit later via Theorem~\ref{thm:fx_projection}.}

\begin{fact}\label{fac:f_projection}
    Suppose that $\v,\w \in \R^n$, and let $P \in \M_n(\R)$ be an orthogonal projection and $f$ be the function defined in Equation~\eqref{eq:cs_func}. Then
    \[
        f(P\v,P\w) \leq f(\v,\w).
    \]
\end{fact}

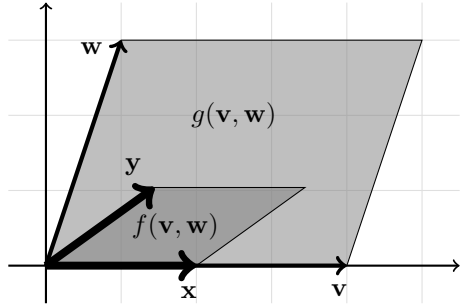
\begin{figure}[!htb]
    \centering
    \begin{tikzpicture}
    	\foreach \x in {0,1,2,3,4,5} \draw[color=gray!25] (\x,-0.5) -- (\x,3.5);
    	\foreach \y in {0,1,2,3} \draw[color=gray!25] (-0.5,\y) -- (5.5,\y);

    	\filldraw[fill=gray,fill opacity=0.5] (0,0) -- (4,0) -- (5,3) -- (1,3) -- cycle;
    	\filldraw[fill=gray,fill opacity=0.5] (0,0) -- (2,0) -- (3.442615,1.039778) -- (1.442615,1.039778) -- cycle;
    	\draw[ultra thick,-to] (0,0) -- (4,0)node[anchor=north,shift={(-0.1,-0.1)}]{$\v$};
    	\draw[ultra thick,-to] (0,0) -- (1,3)node[anchor=east,shift={(-0.1,-0.1)}]{$\w$};
    	\draw[line width=3pt,-to] (0,0) -- (1.442615,1.039778)node[anchor=south east]{$\y$};
    	\draw[line width=3pt,-to] (0,0) -- (2,0)node[anchor=north,shift={(-0.1,-0.1)}]{$\x$};
    	\draw (2.5,2)node{$g(\v,\w)$};
    	\draw (1.7213075,0.519889)node{$f(\v,\w)$};

    	\draw[thick,-to] (-0.5,0) -- (5.5,0);
    	\draw[thick,-to] (0,-0.5) -- (0,3.5);
	\end{tikzpicture}

	\caption{An illustration of how the functions $f$ and $g$ from Equation~\eqref{eq:cs_func} may be interpreted as areas of parallelograms. The angle between $\x$ and $\y$ is half of the angle between $\v$ and $\w$, and the lengths of $\x$ and $\y$ are the square roots of the lengths of $\v$ and $\w$, respectively.}\label{fig:geometric_proof}
\end{figure}

For example, if we choose $P = O$ (the all-zeros matrix) in Fact~\ref{fac:f_projection}, then we recover exactly the Cauchy--Schwarz inequality itself, since the left-hand side of the equation is simply the scalar zero. With another choice of $P$, we obtain the following result:

\begin{theorem}\label{thm:gen_to_matrices}
    Suppose that $X,Y \in \M_n(\R)$, and let $\x = \mathrm{diag}(X)$ and $\y = \mathrm{diag}(Y)$. Then
    \begin{align}\label{eq:gen_to_matrices}
        \|\x\|\|\y\| - \ip{\x}{\y} \leq \|X\|_{\textup{F}}\|Y\|_{\textup{F}} - \ip{X}{Y}_{\textup{F}}.
    \end{align}
\end{theorem}

\begin{proof}
    If we define $\v = \mathrm{vec}(X)$ and $\w = \mathrm{vec}(Y)$ and use the fact that vectorization is an isomorphism from the inner product space $\M_n(\R)$ to the inner product space $\R^n$ (so, for example, we have $\|\v\| = \|X\|_{\textup{F}}$ and $\ip{\v}{\w} = \ip{X}{Y}_{\textup{F}}$), then Fact~\ref{fac:f_projection} tells us that
    \[
        \|P\v\|\|P\w\| - \ip{P\v}{P\w} \leq \|\v\|\|\w\| - \ip{\v}{\w} = \|X\|_{\textup{F}}\|Y\|_{\textup{F}} - \ip{X}{Y}_{\textup{F}}.
    \]
    Now just choose $P$ to be the orthogonal projection onto the span of the diagonal entries of $\v$ and $\w$ (specifically, $P = \mathrm{vec}(I) = \sum_{i=1}^n \mathbf{e}_i\mathbf{e}_i^T \otimes \mathbf{e}_i\mathbf{e}_i^T$), so that $P\v = \x$ and $P\w = \y$.
\end{proof}

To obtain Inequality~\eqref{eq:cs_original} from Theorem~\ref{thm:gen_to_matrices}, we simply choose $X = \v\v^T$ and $Y = \w\w^T$, since then we have (in the notation of that theorem) $\x = \v^2$, $\y = \w^2$, and
\begin{align*}
    \|\v^2\|\|\w^2\| - \ip{\v^2}{\w^2} \leq \|\v\v^T\|_{\textup{F}}\|\w\w^T\|_{\textup{F}} - \ip{\v\v^T}{\w\w^T}_{\textup{F}} = \|\v\|^2\|\w\|^2 - \ip{\v}{\w}^2,
\end{align*}
thus completing the proof.

Another particularly interesting usage of Theorem~\ref{thm:gen_to_matrices} is that it lets us bound the eigenvalues of a pair of symmetric matrices in terms of its diagonal entries. Although the well-known Schur--Horn theorem (see \cite[Theorems~B.1 and~B.2]{marshall11}, for example) already provides such a bound for a single symmetric matrix (and indeed, a complete characterization of which eigenvalues of a symmetric matrix are possible given its diagonal entries), the following corollary seems to be new:

\begin{corollary}\label{cor:gen_to_matrices_eig}
    Suppose that $X,Y \in \M_n(\R)$ are symmetric. Let $\x = \mathrm{diag}(X)$ and $\y = \mathrm{diag}(Y)$, and let $\lambda_X$ and $\lambda_Y$ be vectors of the eigenvalues of $X$ and $Y$, respectively, sorted in opposite order (i.e., one sorted from largest to smallest and the other sorted from smallest to largest). Then
    \begin{align}\label{eq:gen_to_matrices}
        \|\x\|\|\y\| - \ip{\x}{\y} \leq \|\lambda_X\|\|\lambda_Y\| - \ip{\lambda_X}{\lambda_Y}.
    \end{align}
\end{corollary}

\begin{proof}
    Simply recall that for any symmetric matrices $X$ and $Y$, we have $\|X\|_{\textup{F}} = \|\lambda_X\|$ and $\|Y\|_{\textup{F}} = \|\lambda_Y\|$. Theorem~\ref{thm:gen_to_matrices} then tells us that
    \[
        \|\x\|\|\y\| - \ip{\x}{\y} \leq \|\lambda_X\|\|\lambda_Y\| - \ip{X}{Y}_{\textup{F}}.
    \]
    All that remains is to show that $\ip{X}{Y}_{\textup{F}} \geq \ip{\lambda_X}{\lambda_Y}$. This is a well-known inequality in matrix analysis (see \cite[Problem~III.6.14]{Bha97}, for example), so the proof is complete.
\end{proof}

In order to generalize the above corollary to make use of singular values of non-symmetric matrices, we have to move terms around a little bit:

\begin{corollary}\label{cor:gen_to_matrices_svd}
    Suppose that $X,Y \in \M_n(\R)$, let $\x = \mathrm{diag}(X)$ and $\y = \mathrm{diag}(Y)$, and let $\sigma_X$ and $\sigma_Y$ be vectors of the singular values of $X$ and $Y$, respectively, sorted in the same order (e.g., both sorted from largest to smallest). Then
    \begin{align}\label{eq:gen_to_matrices}
        \|\x\|\|\y\| - \ip{\sigma_X}{\sigma_Y} \leq \|\sigma_X\|\|\sigma_Y\| - \ip{\x}{\y}.
    \end{align}
\end{corollary}

\begin{proof}
    Recall that the Frobenius norm has the property that $\|X\|_{\textup{F}} = \|\sigma_X\|$ and $\|Y\|_{\textup{F}} = \|\sigma_Y\|$, so Theorem~\ref{thm:gen_to_matrices}, when applied to the matrices $-X$ and $Y$ instead of $X$ and $Y$, tells us that
    \[
        \|\x\|\|\y\| + \ip{\x}{\y} \leq \|\sigma_X\|\|\sigma_Y\| + \ip{X}{Y}_{\textup{F}}.
    \]
    Now use the fact that $\ip{X}{Y}_{\textup{F}} \leq \ip{\sigma_X}{\sigma_Y}$ (this is sometimes called von Neumann's trace inequality \cite[Theorem~7.4.11]{HJ13}), and rearrange slightly to obtain the desired inequality.
\end{proof}

\section{Tensor Inequalities and Generalizations to Higher Exponents}\label{sec:higher_exponents}

Given a real number $x \geq 1$, consider the function $f : \R^n \times \R^n \rightarrow \R$ defined by
\begin{align}\label{eq:cs_func_x}
    f_x(\v,\w) = \|\v\|^x\|\w\|^x - |\ip{\v}{\w}|^x.
\end{align}
When $x = 1$, the function $f_x$ from Equation~\eqref{eq:cs_func_x} becomes essentially the function $f$ from Equation~\eqref{eq:cs_func} (just with an absolute value around the inner product, which does not change anything substantial, since we can always replace $\w$ by $-\w$ in the original $f$ to force the inner product term to be negative as in $f_x$). When $x = 2$, the function $f_x$ is the square of the Grammian of $\v$ and $\w$---that is, if $A$ is the $n \times 2$ matrix with $\v$ and $\w$ as its columns, then
\[
    f_2(\v,\w) = \|\v\|^2\|\w\|^2 - |\ip{\v}{\w}|^2 = \det(A^TA)
\]
is the square of the area of the parallelogram with sides $\v$ and $\w$ (and is thus equal to the square of the function $g$ from Section~\ref{sec:intuitive_geometric}).

The original Inequality~\eqref{eq:cs_original} can be written in terms of these $f_x$ functions as $f_1(\v^2,\w^2) \leq f_2(\v,\w)$. Our goal in this section is to generalize this inequality to other values of $x$. To this end, first note that since the $2$-norm and associated inner product are unitarily invariant, so is $f_x$. Similarly, we showed in Fact~\ref{fac:f_projection} that $f_1$ is non-increasing under orthogonal projections, and the same is well known to hold for the Grammian and thus $f_2$. We now show that the same holds for each $f_x$, starting with the following lemma:

\begin{lemma}\label{lem:four_exponentials}
    Suppose that $a,b,c,d \geq 0$ are constants with $\max\{a,b\} \geq \max\{c,d\}$ and $a+b \geq c+d$. Let
    \[
        f(x) = a^x + b^x - c^x - d^x.
    \]
    Then $f(x) \geq 0$ for all $x \geq 1$.
\end{lemma}

\begin{proof}
    Without loss of generality, assume that $a\geq b$, $c\geq d$. By scaling $f$, we may assume further that $a+b = 1$ (and therefore $a\geq 1/2$). Moreover, decreasing $c$ and/or $d$ increases the value of $f(x)$, so it suffices to prove the lemma when $c+d = 1$ as well.

    We therefore have $b = 1-a$ and $d = 1-c$, so our goal is to show that
    \begin{align}\label{ineq:ac}
        a^x + (1-a)^x \geq c^x + (1-c)^x \quad \text{for all} \quad x \geq 1.
    \end{align}
    Now consider the function
    \[
        g(a) = a^x + (1-a)^x.
    \]
    Taking the derivative of $g$ with respect to $a$ reveals that
    \[
        g^\prime(a) = x\big(a^{x-1} - (1-a)^{x-1}\big),
    \]
    which is non-negative, since $a \geq 1-a \geq 0$ and $x \geq 1$. It follows that $g$ is monotonically non-decreasing for $a \in [1/2, 1]$. Now observe that $a \geq b$ and $\max\{a,b\} \geq \max\{c,d\}$ imply $a \ge c$, from which it immediately follows that $g(a) \geq g(c)$. This is precisely Inequality~\eqref{ineq:ac}, so we are done.
\end{proof}

\begin{theorem}\label{thm:fx_projection}
    Let $x \geq 1$ be a real number. For any pair of vectors $\v,\w \in \R^n$ and any orthogonal projection $P \in \M_n(\R)$, we have
    \[
        f_x(P\v,P\w) \leq f_x(\v,\w).
    \]
\end{theorem}

\begin{proof}
    Since $f_x$ is unitarily invariant, we may assume without loss of generality that $P$ is diagonal. We may assume further that the range of $P$ is $(n-1)$-dimensional, since every diagonal projection onto a smaller space can be written as a product of diagonal projections onto $(n-1)$-dimensional spaces. Altogether, this lets us assume without loss of generality that $P\v = (v_2, v_3, \ldots, v_{n})$ and $P\w = (w_2, w_3, \ldots, w_n)$.

    We first prove the theorem in the $x = 1$ case, which boils down to
    \begin{align}\label{eq:ineq_x1}
        \|P\v\|\|P\w\| - |\ip{P\v}{P\w}| \leq \|\v\|\|\w\| - |\ip{\v}{\w}|.
    \end{align}
    To prove this inequality, start by noting that the triangle inequality yields
    \begin{align*}
        |\ip{\v}{\w}| & = \left|\sum_{i=1}^n v_iw_i\right| \leq \left|\sum_{i=2}^n v_iw_i\right| + |v_1w_1|,
    \end{align*}
    which we can rearrange slightly to obtain
    \begin{align}\label{eq:ineq_x1_stitch1}
        |\ip{\v}{\w}| - |\ip{P\v}{P\w}| & = \left|\sum_{i=1}^n v_iw_i\right| - \left|\sum_{i=2}^n v_iw_i\right| \leq |v_1w_1|.
    \end{align}
    From here, we apply the Cauchy--Schwarz inequality to the 2-dimensional vectors $(v_1, \mathrm{sign}(v_1)\|P\v\|)$ and $(w_1, \mathrm{sign}(w_1)\|P\w\|)$, where $\mathrm{sign}(z) = z/|z|$ is the sign of $z$, to see that
    \begin{eqnarray*}
        |v_1w_1| + \|P\v\|\|P\w\| &\leq& \sqrt{v_1^2 + \|P\v\|^2} \sqrt{w_1^2 + \|P\w\|^2} \\
        &=& \sqrt{v_1^2 + v_2^2 +\cdots+ v_n^2} \sqrt{w_1^2 + w_2^2 +\cdots + w_n^2} \\
        &=& \|\v\|\|\w\|.
    \end{eqnarray*}
    Rearranging slightly then gives us
    \begin{align}\label{eq:ineq_x1_stitch2}
        |v_1w_1| \leq \|\v\|\|\w\| - \|P\v\|\|P\w\|.
    \end{align}
    Stitching together Inequalities~\eqref{eq:ineq_x1_stitch1} and~\eqref{eq:ineq_x1_stitch2} gives exactly Inequality~\eqref{eq:ineq_x1} as desired, thus completing the proof of the theorem in the $x = 1$ case.

    To establish that the theorem holds for $x > 1$ as well, notice how it follows from the $x = 1$ case that when $a = \|\v\|\|\w\|$, $b = |\ip{P\v}{P\w}|$, $c = \|P\v\|\|P\w\|$, and $d = |\ip{\v}{\w}|$, we have $a+b \geq c+d$. We can thus apply Lemma~\ref{lem:four_exponentials} (the inequality $a \geq \max\{c,d\}$ is straightforward) to obtain $a^x + b^x \geq c^x + d^x$ for all $x \geq 1$ as well. Plugging in $a,b,c,d$ and rearranging then gives exactly the statement of the theorem.
\end{proof}

Much like we used Fact~\ref{fac:f_projection} to prove Theorem~\ref{thm:gen_to_matrices}, we can now use Theorem~\ref{thm:fx_projection} to prove another generalization of Inequality~\eqref{eq:cs_original}. First, however, we need to introduce a bit more notation and terminology. For any pair of vectors $\x, \y\in \R^n$, define $\x\odot \y$ to be the Hadamard (entrywise) product of $\x$ and $\y$. Let $p \geq 1$ be an integer and $P$ be the orthogonal projection onto
\begin{align*}
    \mathrm{span}\big\{ \e_j^{\otimes p} : 1 \leq j \leq n \big\} \subset (\R^n)^{\otimes p}.
\end{align*}
(For example, if $p = 2$ and we identify $\M_n(\R)$ with $\R^n \otimes \R^n$ via vectorization, then $P$ is exactly the projection onto the diagonal of matrix in $\M_n(\R)$.)

It is straightforward to show that $P(\otimes_{j=1}^p \x_j)$ is a vector with at most $n$ nonzero entries, which are exactly the entries of $\x_1 \odot \cdots \odot \x_p$. If we make use of this orthogonal projection $P$ and the vectors $\v = \otimes_{j=1}^p \x_j$ and $\w = \otimes_{j=1}^p \y_j$ in Theorem~\ref{thm:fx_projection}, we immediately obtain the following result:
\begin{corollary}\label{cor:fp_projection_diagonal}
    Let $x \geq 1$ be a real number and $p \geq 1$ be an integer. For any collection of vectors $\x_1,\y_1,\ldots,\x_p,\y_p \in \R^n$, we have
    \[
        \| \x_1 \odot \cdots \odot \x_p\|^x\| \y_1 \odot \cdots \odot \y_p\|^x - |\ip{\x_1 \odot \cdots \odot \x_p}{\y_1 \odot \cdots \odot \y_p}|^x \leq \prod_{j=1}^p \|\x_j\|^x\|\y_j\|^x - \prod_{j=1}^p |\ip{\x_j}{\y_j}|^x.
    \]
\end{corollary}

The above corollary has numerous special cases of interest. For example, if we choose $\x_j = \x$ and $\y_j = \y$ for all $1 \leq j \leq p$ then we get the following further specialization in the $x = 1$ case:

\begin{corollary}\label{cor:equal_tensors}
    Let $p \geq 1$ be an integer. For any pair of vectors $\x,\y \in \R^n$, we have
    \begin{align}\label{ineq:integer_p}
        \| \x^p \|\| \y^p \| - |\ip{\x^p}{\y^p}| \leq \|\x\|^{p}\|\y\|^{p} - |\ip{\x}{\y}|^{p}.
    \end{align}
\end{corollary}

\noindent If we specialize further to $p = 2$, then we obtain Inequality~\eqref{eq:cs_original}.

\subsection{A Sum-of-Squares Decomposition}\label{sec:sum_of_squares}

In this subsection, we provide another proof of Corollary~\ref{cor:equal_tensors}, which we believe to be of independent interest, in the case that $p \geq 2$ is an even integer.

Recall that the original proof of Inequality~\eqref{eq:cs_original} from \cite{MO_Lemma,JM19} was based on Lagrange's identity, which states that
\begin{align}\label{eq:lagrange}
    \|\v\|^2\|\w\|^2 - \ip{\v}{\w}^2 = \frac{1}{2}\sum_{i=1}^{n} \sum_{j=1,j\neq i}^n (v_i w_j - v_j w_i)^2.
\end{align}
This inequality is interesting for the fact that it implies that the polynomial $\|\v\|^2\|\w\|^2 - \ip{\v}{\w}^2$ (which may be regarded as a degree-$4$ homogeneous polynomial in the $2n$ variables $v_1$, $\ldots$, $v_n$, $w_1$, $\ldots$, $w_n$) is not just positive semidefinite (a statement that is equivalent to the Cauchy--Schwarz inequality) but can even be written as a sum of squares (a stronger property).

The question of whether a given polynomial can be written as a sum of squares dates back to 1888 \cite{Hilbert} and is the 17th problem in Hilbert's famous list of important open mathematical problems (compiled in 1900). Though there is much interest in this problem from a purely mathematical standpoint, more recent interest in the sum-of-squares problem relates to optimizing multivariable polynomials via semidefinite programming \cite{parrilo2000}. In particular, determining whether a given polynomial is positive semidefinite is NP-hard; however, identifying a polynomial as a sum of squares immediately certifies this property.

Lagrange's identity shows that the polynomial $\|\v\|^2\|\w\|^2 - \ip{\v}{\w}^2$ is not just positive semidefinite but can also be written as a sum of squares. We now provide a generalization of this identity demonstrating that the same is true of the polynomial $\|\v\|^p\|\w\|^p - \ip{\v}{\w}^p$ whenever $p \geq 2$ is an even integer:

\begin{theorem}\label{thm:sumofsquares}
    If $k \geq 1$ is an integer, then
    \begin{equation}\label{eq:SOS}
        \|\v\|^{2k}\|\w\|^{2k} - \ip{\v}{\w}^{2k} = \frac{1}{2}\sum_{\substack{a_1 + \cdots + a_n = k\\ b_1 + \cdots + b_n = k}}{k\choose a_1, a_2, \ldots, a_n}{k\choose b_1, b_2, \ldots, b_n}\left(\prod_{i = 1}^{n} v_i^{a_i}w_i^{b_i} - \prod_{i = 1}^{n} v_i^{b_i}w_i^{a_i}\right)^2,
    \end{equation}
    where the sum is over all $n$-tuples $(a_1, \ldots, a_n)$ and $(b_1, \ldots, b_n)$ satisfying $a_i,b_i \geq 0$ for all $1 \leq i \leq n$ and $a_1 + \cdots + a_n = b_1 + \cdots + b_n = k$.
\end{theorem}

\begin{proof}
    The left-hand side of Equation~\eqref{eq:SOS} is
    \begin{align}
        \|\v\|^{2k}\|\w\|^{2k} - \ip{\v}{\w}^{2k}\nonumber & = \left(\sum_{i=1}^n v_i^2\right)^k\left(\sum_{i=1}^nw_i^2\right)^k - \left(\sum_{i=1}^n v_iw_i\right)^k\left(\sum_{i=1}^n v_iw_i\right)^k\nonumber \\
        & = \left(\sum_{a_1 + \cdots + a_n = k}{k\choose a_1,\ldots, a_n}\prod_{i = 1}^{n} v_i^{2a_i}\right)\left(\sum_{b_1 + \cdots + b_n = k} {k\choose b_1,\ldots, b_n}\prod_{i = 1}^{n} w_i^{2b_i}\right)\nonumber \\
        & \qquad - \left(\sum_{a_1 + \cdots + a_n = k}{k\choose a_1, \ldots, a_n}\prod_{i = 1}^{n} v_i^{a_i}w_i^{a_i}\right)\left(\sum_{b_1 + \cdots + b_n =k}{k\choose b_1,\ldots, b_n}\prod_{i = 1}^{n} v_i^{b_i}w_i^{b_i}\right)\nonumber \\
        & = \sum_{\substack{a_1 + \cdots + a_n = k\\ b_1 + \cdots + b_n = k}} {k\choose a_1,\ldots, a_n}{k\choose b_1,\ldots, b_n}\left(\prod_{i = 1}^{n} v_i^{2a_i}\right)\left(\prod_{i = 1}^{n} w_i^{2b_i}\right)\nonumber \\
        & \qquad - \sum_{\substack{a_1 + \cdots + a_n = k\\ b_1 + \cdots + b_n = k}} {k\choose a_1,\ldots, a_n}{k\choose b_1,\ldots, b_n}\left(\prod_{i = 1}^{n} v_i^{a_i}w_i^{a_i}\right)\left(\prod_{i = 1}^{n} v_i^{b_i}w_i^{b_i}\right)\nonumber \\
        & = \sum_{\substack{a_1 + \cdots + a_n = k\\ b_1 + \cdots + b_n = k}} {k\choose a_1,\ldots, a_n}{k\choose b_1,\ldots, b_n}\prod_{i = 1}^{n} v_i^{2a_i}w_i^{2b_i}\nonumber \\
        & \qquad - \sum_{\substack{a_1 + \cdots + a_n = k\\ b_1 + \cdots + b_n = k}} {k\choose a_1,\ldots, a_n}{k\choose b_1,\ldots, b_n}\prod_{i = 1}^{n} v_i^{a_i + b_i}w_i^{a_i + b_i}.\label{eq:LHS}
    \end{align}
    On the other hand, the right-hand side of Equation~\eqref{eq:SOS} is 
    \begin{align}
        & \frac{1}{2}\sum_{\substack{a_1 + \cdots + a_n = k\\ b_1 + \cdots + b_n = k}}{k\choose a_1,\ldots, a_n}{k\choose b_1,\ldots, b_n}\left(\prod_{i = 1}^{n} v_i^{a_i}w_i^{b_i} - \prod_{i = 1}^{n} v_i^{b_i}w_i^{a_i}\right)^2 \nonumber \\
        & \qquad = \frac{1}{2}\sum_{\substack{a_1 + \cdots + a_n = k\\ b_1 + \cdots + b_n = k}}{k\choose a_1,\ldots, a_n}{k\choose b_1,\ldots, b_n}\left(\prod_{i = 1}^{n} v_i^{2a_i}w_i^{2b_i} + \prod_{i = 1}^{n} v_i^{2b_i}w_i^{2a_i} - 2\prod_{i = 1}^{n} v_i^{a_i}w_i^{b_i} v_i^{b_i}w_i^{a_i}\right)\nonumber \\
        & \qquad = \frac{1}{2}\sum_{\substack{a_1 + \cdots + a_n = k\\ b_1 + \cdots + b_n = k}}{k\choose a_1,\ldots, a_n}{k\choose b_1,\ldots, b_n}\left(\prod_{i = 1}^{n} v_i^{2a_i}w_i^{2b_i} + \prod_{i = 1}^{n} v_i^{2b_i}w_i^{2a_i}\right)\nonumber \\
        & \qquad \qquad - \frac{1}{2}\sum_{\substack{a_1 + \cdots + a_n = k\\ b_1 + \cdots + b_n = k}}{k\choose a_1,\ldots, a_n}{k\choose b_1,\ldots, b_n}\left(2\prod_{i = 1}^{n} v_i^{a_i + b_i}w_i^{b_i+a_i}\right)\nonumber \\
        & \qquad = \sum_{\substack{a_1 + \cdots + a_n = k\\ b_1 + \cdots + b_n = k}}{k\choose a_1,\ldots, a_n}{k\choose b_1,\ldots, b_n}\prod_{i = 1}^{n} v_i^{2a_i}w_i^{2b_i}\nonumber \\
        & \qquad \qquad - \sum_{\substack{a_1 + \cdots + a_n = k\\ b_1 + \cdots + b_n = k}}{k\choose a_1,\ldots, a_n}{k\choose b_1,\ldots, b_n}\prod_{i = 1}^{n} v_i^{a_i + b_i}w_i^{b_i+a_i}.\label{eq:RHS}
    \end{align}
    A comparison of the quantities at the end of Equations~\eqref{eq:LHS} and~\eqref{eq:RHS} reveals that they are equal, thus completing the proof.
\end{proof}

We now use the above sum-of-squares decomposition to provide an alternate proof of Corollary~\ref{cor:equal_tensors} in the case that $p \geq 2$ is an even integer:

\begin{corollary}\label{cor:equal_tensors_from_lagrange}
    Let $k \geq 1$ be an integer. For any pair of vectors $\v,\w \in \R^n$, we have
    \begin{align}\label{ineq:integer_p}
        \| \v^{2k} \|\| \w^{2k} \| - \ip{\v^{2k}}{\w^{2k}} \leq \| \v^{k} \|^2\| \w^{k} \|^2 - \ip{\v^{k}}{\w^{k}}^2 \leq \|\v\|^{2k}\|\w\|^{2k} - \ip{\v}{\w}^{2k}.
    \end{align}
\end{corollary}

\begin{proof}
    The left inequality follows simply from applying Inequality~\eqref{eq:cs_original} to the vectors $\v^k$ and $\w^k$. To obtain the right inequality, we start by applying Lagrange's identity (Equation~\eqref{eq:lagrange}) to $\v^k$ and $\w^k$:
    \begin{align}\label{eq:vkwk_lagrange}
        \| \v^{k} \|^2\| \w^{k} \|^2 - \ip{\v^{k}}{\w^{k}}^2 = \frac{1}{2}\sum_{i=1}^{n} \sum_{j=1,j\neq i}^n (v_i^k w_j^k - v_j^k w_i^k)^2.
    \end{align}

    Now notice that each term $(v_i^k w_j^k - v_j^k w_i^k)^2$ in Equation~\eqref{eq:vkwk_lagrange} arises as the term
    \[
        \left(\prod_{i = 1}^{n} v_i^{a_i}w_i^{b_i} - \prod_{i = 1}^{n} v_i^{b_i}w_i^{a_i}\right)^2
    \]
    in Equation~\eqref{eq:SOS} when $a_i = k$, $b_j = k$, and all the other $a$'s and $b$'s are $0$. Since the formula given in Equation~\eqref{eq:SOS} consists precisely of the terms in Equation~\eqref{eq:vkwk_lagrange} plus additional non-negative terms, it must be the case that $\| \v^{k} \|^2\| \w^{k} \|^2 - \ip{\v^{k}}{\w^{k}}^2 \leq \|\v\|^{2k}\|\w\|^{2k} - \ip{\v}{\w}^{2k}$. This is exactly the right inequality that we set out to prove, so we are done.
\end{proof}

\section{Generalization to More Vectors}\label{sec:quantum}

In this section, we generalize Inequality~\eqref{eq:cs_original} to more than two vectors. To do this, it is useful for us to clarify how this inequality was originally derived, since this line of reasoning is the very same one that we will be generalizing. This method was implicitly presented in \cite{JM19} but never explicitly described; moreover, it relies on several nontrivial facts from quantum information theory that we now introduce.

\subsection{Inequality~\eqref{eq:cs_original} via Separability of Quantum States}\label{sec:quantum_original}

A positive semidefinite matrix $X \in \M_n(\C) \otimes \M_n(\C)$ is called \emph{separable} if it has a decomposition of the form
\begin{align}\label{eq:separable_bipartite}
    X = \sum_{j=1}^r Y_j \otimes Z_j,
\end{align}
where $r$ is a positive integer and $\{Y_j\}, \{Z_j\} \subset \M_n(\C)$ are finite sets of positive semidefinite matrices (note that in this context, positive semidefinite matrices are necessarily Hermitian). Determining separability of a matrix is NP-hard \cite{Gha10,Gur03}---indeed, this is a central research problem in quantum information theory \cite{GT09,HHH09}.

It is straightforward to show that if $X \in \M_n(\C) \otimes \M_n(\C)$ is separable, then so is $(A \otimes B)X(A \otimes B)^*$ for every $A, B \in \M_n(\C)$. Since the set of separable matrices is convex, it follows that separability of $X$ implies separability of
\begin{align}\label{eq:local_diag_twirl}
    T(X) := \int_U (U \otimes U)X(U \otimes U)^* \, \mathrm{d}U,
\end{align}
where the integration is with respect to Haar measure over the set of diagonal unitary matrices $U$. We call $T(X)$ the \emph{local diagonal unitary twirl} of $X$. If we index the entries of $X$ as
\begin{align*}
    [X]_{(i_1,j_1),(i_2,j_2)} := (\e_{i_1}\otimes \e_{i_2})^* X(\e_{j_1}\otimes \e_{j_2}),
\end{align*}
then straightforward computation shows that for a diagonal unitary matrix $U$ with diagonal entries $u_1$, $u_2$, $\ldots$, $u_n$ we have
\begin{align*}
    \big[(U \otimes U)X(U \otimes U)^*\big]_{(i_1,j_1),(i_2,j_2)} = u_{i_1}u_{i_2}\overline{u_{j_1}u_{j_2}}[X]_{(i_1,j_1),(i_2,j_2)}.
\end{align*}
Integrating this quantity then reveals that the following explicit formula for the local diagonal unitary twirl:
\[
    T(X) = \sum_{i,j=1}^n [X]_{(i,i),(j,j)}\e_i\e_j^* \otimes \e_i\e_j^* + \sum_{i \neq j=1}^n [X]_{(i,j),(j,i)}\e_i\e_j^* \otimes \e_j\e_i^*
\]
(see \cite{SN21} for more detailed calculations of this type).

One method of showing that a matrix is not separable is the \emph{realignment criterion} \cite{CW03,Rud03}, which states that if $X \in \M_n(\C) \otimes \M_n(\C)$ is separable and $R : \M_n(\C) \otimes \M_n(\C) \rightarrow \M_n(\C) \otimes \M_n(\C)$ is the linear map defined by
\begin{align}\label{eq:realignment_map}
    R\big(\e_i\e_j^T \otimes \e_k\e_\ell^T\big) = \e_i\e_k^T \otimes \e_j\e_\ell^T,
\end{align}
then $\|R(X)\|_{\textup{tr}} \leq \tr(X)$.\footnote{In practice, the contrapositive of this statement is used: it is straightforward to compute $\|R(X)\|_{\textup{tr}}$, and if this quantity is larger than $\tr(X)$, then we know that $X$ cannot be separable.} If we apply the realignment criterion of the local diagonal unitary twirl of $X = \v\v^* \otimes \w\w^*$ (which is clearly separable), we learn that for all vectors $\v,\w \in \R^n$, we have
\begin{align}\label{eq:realign_ineq_vw}
    \|R(T(\v\v^T \otimes \w\w^T))\|_{\textup{tr}} \leq \tr(T(\v\v^T \otimes \w\w^T)) = \tr(\v\v^T \otimes \w\w^T) = \|\v\|^2\|\w\|^2.
\end{align}
The trace norm on the left-hand side is in fact straightforward to compute, since the matrix $R(T(\v\v^T \otimes \w\w^T))$ is (up to permutation of the rows and columns) block diagonal with diagonal blocks equal to $(\v^2)(\w^2)^T$ and $v_iv_kw_iw_k$ for $1 \leq i \neq k \leq n$. It therefore follows that
\begin{align}\begin{split}\label{eq:realign_explicit}
    \big\|\v^2 (\w^2)^T\big\|_{\textup{tr}} + \sum_{i \neq k} |v_i v_k w_i w_k| \leq \|\v\|^2\|\w\|^2.
\end{split}\end{align}
Using the fact that $\|\v^2 (\w^2)^T\|_{\textup{tr}} = \|\v^2\|\|\w^2\|$, rearranging slightly gives exactly Inequality~\eqref{eq:cs_original}, exactly as desired. (We could write an analogous inequality for all vectors $\v,\w \in \C^n$ instead, but this would not actually be any more general---terms like $\v^2$ would simply replaced by terms like $\v \odot \overline{\v}$.)

\subsection{Multipartite Tensor Products and More Vectors}\label{sec:quantum_multi}

To generalize the argument from Section~\ref{sec:quantum_original} to three or more vectors, we work in the tensor product spaces $(\C^n)^{\otimes p}$ and $(\M_n(\C))^{\otimes p}$, where $p > 2$. We say that a positive semidefinite matrix $X \in (\M_n(\C))^{\otimes p}$ is \emph{separable} if it has a decomposition of the form
\[
    X = \sum_{j=1}^r Y_j^{(1)} \otimes Y_j^{(2)} \otimes \cdots \otimes Y_j^{(p)},
\]
where $r$ is a positive integer and $\{Y_j^{(k)}\} \subset \M_n(\C)$ is a finite set of positive semidefinite matrices (this recovers the definition of separability from Equation~\eqref{eq:separable_bipartite} when $p = 2$).

There are many possible ways to generalize the realignment criterion to this setting. Given a permutation $\sigma \in \S_{2p}$, we say that the \textit{realignment map associated with $\sigma$} is the linear map $R_\sigma : (\M_n(\C))^{\otimes p} \rightarrow (\M_n(\C))^{\otimes p}$ defined by
\begin{align}\label{eq:multipartite_realign}
    R_\sigma\big(\v_1\v_{p+1}^T \otimes \v_2\v_{p+2}^T \otimes \cdots \otimes \v_p\v_{2p}^T\big) = \v_{\sigma(1)}\v_{\sigma(p+1)}^T \otimes \v_{\sigma(2)}\v_{\sigma(p+2)}^T \otimes \cdots \otimes \v_{\sigma(p)}\v_{\sigma(2p)}^T
\end{align}
for all $\v_1,\v_2,\ldots,\v_{2p} \in \C^n$. (In the case that $p = 2$ and $\sigma = (1,3,2,4)$, this is just the usual realignment map from Equation~\eqref{eq:realignment_map}.) It is known \cite{HHH06} that $\|R_\sigma(X)\|_{\textup{tr}} \leq \tr(X)$ for all permutations $\sigma \in \S_{2p}$ and all separable $X \in (\M_n(\C))^{\otimes p}$, thus generalizing the realignment criterion.

To generalize the local diagonal unitary twirl from Equation~\eqref{eq:local_diag_twirl} to the $p > 2$ case, we simply take higher tensor powers of the diagonal unitary matrix $U$:
\begin{align*}
    T(X) := \int_U (U^{\otimes p})X(U^{\otimes p})^* \, \mathrm{d}U,
\end{align*}
for all $X \in (\M_n(\C))^{\otimes p}$ (here the integration is with respect to Haar measure over the set of diagonal unitary matrices $U$). If we index the entries of $X$ via
\begin{align*}
    [X]_{(i_1,j_1),\ldots,(i_p,j_p)} := (\e_{i_1}\otimes \cdots \otimes \e_{i_p})^* X(\e_{j_1}\otimes \cdots \otimes \e_{j_p}),
\end{align*}
then straightforward computation shows that
\begin{align*}
    \big[(U^{\otimes p})X(U^{\otimes p})^*\big]_{(i_1,j_1),\ldots,(i_p,j_p)} = (u_{i_1}\cdots u_{i_p})\overline{(u_{j_1}\cdots u_{j_p})}[X]_{(i_1,j_1),\ldots,(i_p,j_p)}.
\end{align*}
Integrating this quantity then reveals that
\[
    \big[T(X)\big]_{(i_1,j_1),\ldots,(i_p,j_p)} = \begin{cases}
        [X]_{(i_1,j_1),\ldots,(i_p,j_p)} & \text{if} \ \{i_1,\ldots,i_p\} = \{j_1,\ldots,j_p\} \\
        0 & \text{otherwise},
    \end{cases}
\]
where $\{i_1,\ldots,i_p\}$ and $\{j_1,\ldots,j_p\}$ are multisets. Putting all of this together immediately gives the following generalization of Inequality~\eqref{eq:cs_original}:

\begin{theorem}\label{thm:generalize_many_vecs}
    Suppose that $\v_1,\ldots,\v_p \in \C^n$ are vectors and $\sigma \in S_{2p}$ is a permutation. Then
    \begin{align}\label{eq:generalize_more_vecs_ineq}
        \big\|R_{\sigma}\big(T(\v_1
        \v_1^*\otimes \cdots \otimes \v_p\v_p^*)\big)\big\|_{\textup{tr}} \leq \prod_{j=1}^p \big\|\v_j\big\|^2.
    \end{align}
\end{theorem}

In the special case that $p = 2$ and $\sigma = (1,3,4,2)$, the left-hand side of Inequality~\eqref{eq:generalize_more_vecs_ineq} simplifies to $\|\v_1^2\|\|\v_2^2\| - \ip{\v_1^2}{\v_2^2} + \ip{\v_1}{\v_2}^2$, thus recovering exactly Inequality~\eqref{eq:cs_original}. In fact, this is the only nontrivial inequality that arises from Theorem~\ref{thm:generalize_many_vecs} in the $p = 2$ case: every one of the $4! = 24$ possible permutations results in one of two values for the left-hand side of Inequality~\eqref{eq:generalize_more_vecs_ineq}:
\begin{itemize}
    \item $\|\v_1\|^2\|\v_2\|^2$ (e.g., if $\sigma$ is the identity permutation). This results in the trivial inequality $\|\v_1\|^2\|\v_2\|^2 \leq \|\v_1\|^2\|\v_2\|^2$.

    \item $\|\v_1^2\|\|\v_2^2\| - \ip{\v_1^2}{\v_2^2} + \ip{\v_1}{\v_2}^2$ (e.g., if $\sigma = (1,3,4,2)$). This results in Inequality~\eqref{eq:cs_original}.
\end{itemize}

For larger values of $p$, new inequalities result from choosing one of the $(2p)!$ possible permutations in Theorem~\ref{thm:generalize_many_vecs}. For example, in the $p = 3$ case we obtain the following:

\begin{corollary}\label{cor:tripartite_cs}
    Suppose that $\v,\w,\x \in \R^n$ are vectors. Then
    \begin{align*}
        & \sum_{j=1}^n \sqrt{\big(v_j^2\|\v \odot \w\|^2 + w_j^2\|\v^2\|^2 - v_j^4 w_j^2\big)\big(x_j^2\|\x \odot \w\|^2 + w_j^2\|\x^2\|^2 - x_j^4 w_j^2\big)} \\
        & + \sum_{i=1}^n v_ix_i\sqrt{\sum_{j=1,j\neq i}^n\sum_{k=1,k\neq i,k \neq j}^n v_j^2 w_k^2}\sqrt{\sum_{j=1,j\neq i}^n\sum_{k=1,k\neq i,k \neq j}^n x_j^2 w_k^2} \\
        & \leq \|\v\|^2\|\w\|^2\|\x\|^2 + \|\v \odot \w \odot \x\|^2 - \ip{\v \odot \w}{\x \odot \w}\ip{\v}{\x}.
    \end{align*}
\end{corollary}

\begin{proof}
    Choose the permutation $\sigma = (6,5,3,4,2,1)$. Direct computation then reveals that the matrix $R_{\sigma}\big(T(\v\v^T\otimes \w\w^T \otimes \x\x^T)\big)$ is (up to permutation of its rows and columns) block diagonal with diagonal blocks as follows:

    \begin{itemize}
        \item $n$ blocks of size $(2n-1) \times (2n-1)$ each. If we define, for each $j \in \{1,2,\ldots,n\}$, vectors by
        \begin{align*}
            \mathbf{a}_j & = \big(v_j(v_1w_1),\ldots,v_j(v_nw_n), \ w_jv_1^2,\ldots,w_jv_{j-1}^2,w_jv_{j+1}^2,\ldots,w_jv_n^2\big) \quad \text{and} \\
            \mathbf{b}_j & = \big(x_j(x_1w_1),\ldots,x_j(x_nw_n), \ w_jx_1^2,\ldots,w_jx_{j-1}^2,w_jx_{j+1}^2,\ldots,w_jx_n^2\big),
        \end{align*}
        then these blocks are equal to $\a_j\b_j^T$. Since this block has rank~$1$, its trace norm is
        \[
            \big\|\a_j\b_j^T\big\|_{\textup{tr}} = \|\a_j\|\|\b_j\| = \sqrt{\big(v_j^2\|\v \odot \w\|^2 + w_j^2\|\v^2\|^2 - v_j^4 w_j^2\big)\big(x_j^2\|\x \odot \w\|^2 + w_j^2\|\x^2\|^2 - x_j^4 w_j^2\big)}.
        \]
        \item $n$ blocks of size $(n-1)(n-2) \times (n-1)(n-2)$ each. We define, for each $i \in \{1,2,\ldots,n\}$, the vector $\y_i \in \R^{(n-1)(n-2)}$ to have entries equal to all possible products of the form $v_iv_jw_k$ with $1 \leq j,k \leq n$ and none of $i$, $j$, or $k$ equal to each other. We similarly define $\z_i \in \R^{(n-1)(n-2)}$ to have entries equal to all possible products of the form $x_ix_j^2w_k^2$. Then these blocks are equal to $\y_i\z_i^T$ and have trace norm equal to
        \[
            \big\|\y_i\z_i^T\big\|_{\textup{tr}} = \|\y_i\|\|\z_i\| = |v_ix_i|\sqrt{\sum_{j=1,j\neq i}^n\sum_{k=1,k\neq i,k \neq j}^n v_j^2 w_k^2}\sqrt{\sum_{j=1,j\neq i}^n\sum_{k=1,k\neq i,k \neq j}^n x_j^2 w_k^2}.
        \]

        \item $n(n-1)$ blocks of size $1 \times 1$ each. These diagonal entries are equal to $v_iw_iv_jw_jx_i^2$ for all $1 \leq i \neq j \leq n$, so they contribute a total of
        \[
            \sum_{i=1}^n \sum_{j=1,j\neq i}^n x_i^2|v_iw_iv_jw_j|
        \]
        to the trace norm of $R_{\sigma}\big(T(\v\v^T\otimes \w\w^T \otimes \x\x^T)\big)$.
    \end{itemize}
    Since the trace norm is unitarily invariant (and thus invariant under permutations of its rows and columns) and the trace norm of a block-diagonal matrix equals the sum of the trace norms of its diagonal blocks, we have
    \begin{align*}
        \big\|R_{\sigma}\big(T(\v\v^T\otimes \w\w^T \otimes \x\x^T)\big)\big\|_{\textup{tr}} & = \sum_{j=1}^n \sqrt{\big(v_j^2\|\v \odot \w\|^2 + w_j^2\|\v^2\|^2 - v_j^4 w_j^2\big)\big(x_j^2\|\x \odot \w\|^2 + w_j^2\|\x^2\|^2 - x_j^4 w_j^2\big)} \\
        & + \sum_{i=1}^n |v_ix_i|\sqrt{\sum_{j=1,j\neq i}^n\sum_{k=1,k\neq i,k \neq j}^n v_j^2 w_k^2}\sqrt{\sum_{j=1,j\neq i}^n\sum_{k=1,k\neq i,k \neq j}^n x_j^2 w_k^2} \\
        & + \sum_{i=1}^n \sum_{j=1,j\neq i}^n x_i^2|v_iw_iv_jw_j|.
    \end{align*}
    Plugging this expression into Theorem~\ref{thm:generalize_many_vecs} and rearranging slightly produces the desired corollary.
\end{proof}

In fact, the inequality described by Corollary~\ref{cor:tripartite_cs} is essentially the only one that can be obtained from Theorem~\ref{thm:generalize_many_vecs} in the $p = 3$ case. There are $6! = 720$ different permutations that we can choose in that theorem when $p = 3$, but direct computation shows that they all result in one of the following:
\begin{itemize}
    \item The trivial inequality $\|\v\|^2\|\w\|^2\|\x\|^2 \leq \|\v\|^2\|\w\|^2\|\x\|^2$ (e.g., if $\sigma$ is the identity permutation);

    \item The inequality of Corollary~\ref{cor:tripartite_cs}; or

    \item The inequality of Corollary~\ref{cor:tripartite_cs} with the roles of the vectors $\v$, $\w$, and $\x$ permuted.
\end{itemize}

\section{Generalization to Non-Integer Exponents}\label{sec:non_integer}

Thanks to Corollary~\ref{cor:equal_tensors}, we know that Inequality~\eqref{eq:cs_original} generalizes to positive integer exponents $p$ other than just $p = 2$. It is, therefore, only natural to ask whether the same is true of (positive) non-integer exponents.

Before proceeding, we note that throughout this section we will restrict our attention to vectors $\v$ and $\w$ with non-negative real entries---otherwise, the vectors $\v^p$ and $\w^p$ would have complex entries, making it difficult to discuss inequalities involving these quantities. We could put absolute values around quantities like $\ip{\v^p}{\w^p}$ and $\ip{\v}{\w}^p$ to get around this issue, but for simplicity, we focus our attention on real-valued quantities only.

It is not difficult to find examples to show that if $p \in (0,1)$, then we cannot in general say that $\|\v^p\|\|\w^p\| - \ip{\v^p}{\w^p}$ is bigger or smaller than $\|\v\|^p\|\w\|^p - \ip{\v}{\w}^p$. For instance, if $\v = (1,1)$ and $\w = (1,2)$ then we have
\[
    \|\v^p\|\|\w^p\| - \ip{\v^p}{\w^p} = \sqrt{2 + 2^{2p+1}} - (1 + 2^p)
\]
and
\[
     \|\v\|^p\|\w\|^p - \ip{\v}{\w}^p = (\sqrt{10})^p - 3^p.
\]
Standard calculus techniques can be used to show that $(\sqrt{10})^p - 3^p > \sqrt{2 + 2^{2p+1}} - (1 + 2^p)$ for all $p \in (0,1)$, so for this choice of $\v$ and $\w$, we have $\|\v\|^p\|\w\|^p - \ip{\v}{\w}^p > \|\v^p\|\|\w^p\| - \ip{\v^p}{\w^p}$ for all $p \in (0,1)$. On the other hand, if $\v = (1,1)$ and $\w = (0,1)$, then we have
\[ 
    \|\v^p\|\|\w^p\| - \ip{\v^p}{\w^p} = \sqrt{2} - 1
\]
and 
\[
    \norm{\v}^p\norm{\w}^p - \ip{\v}{\w}^p = (\sqrt{2})^p - 1.
\]
Since $(\sqrt{2})^p < \sqrt{2}$ for all $p \in (0,1)$, it follows that for this choice of $\v$ and $\w$, we also have $\|\v\|^p\|\w\|^p - \ip{\v}{\w}^p < \|\v^p\|\|\w^p\| - \ip{\v^p}{\w^p}$ for all $p \in (0,1)$.

Similar examples show that $\|\v^p\|\|\w^p\| - \ip{\v^p}{\w^p}$ and $\|\v\|^p\|\w\|^p - \ip{\v}{\w}^p$ are incomparable when $p \in (1,2)$ as well. For example, if $\v = (1,1)$ and $\w = (0,1)$, then we get
\[
    \|\v\|^p\|\w\|^p - \ip{\v}{\w}^p = (\sqrt{2})^p - 1 > \sqrt{2} - 1 = \|\v^p\|\|\w^p\| - \ip{\v^p}{\w^p}
\]
for all $p \in (1,2)$. Examples demonstrating the opposite inequality for $p \in (1,2)$ are slightly more complicated, but they exist:

\begin{example}\label{exam:p12}
    Define the vectors $\v = (1,1)$ and $\w_\varepsilon = (1,1+\varepsilon)$. We claim that for each $p \in (1,2)$, there exists an $\varepsilon > 0$ sufficiently small so that
    \begin{align}\label{ineq:p12_counter}
        \|\v\|^p\|\w_\varepsilon\|^p - \ip{\v}{\w_\varepsilon}^p < \|\v^p\|\|\w_\varepsilon^p\| - \ip{\v^p}{\w_\varepsilon^p}.
    \end{align}
    
    To verify this claim, recall that the binomial series
    \[
        (1 + x)^p = \sum_{n=0}^\infty \binom{p}{n} x^n
    \]
    converges whenever $|x| < 1$ (and we extend the binomial coefficients to non-integer values of $p$ in the usual way via $\binom{p}{n} = p(p-1)\cdots(p-n+1)/n!$). Then we can compute binomial series expansions for each of the four terms appearing in Inequality~\eqref{ineq:p12_counter} as follows:
    \begin{align*}
        \ip{\v^p}{\w_\varepsilon^p} & = 1 + (1 + \varepsilon)^p = 1 + \sum_{n=0}^\infty \binom{p}{n} \varepsilon^n, \\
        \ip{\v}{\w_\varepsilon}^p & = (2 + \varepsilon)^p = 2^p\left(1 + \frac{\varepsilon}{2}\right)^p = 2^p\sum_{n=0}^\infty \binom{p}{n} \left(\frac{\varepsilon}{2}\right)^n, \\
        \|\v\|^p\|\w_\varepsilon\|^p & = \big(\sqrt{1^2 + 1^2}\big)^p\big(\sqrt{1^2 + (1 + \varepsilon)^2}\big)^p\\
        & = 2^p\left(1 + \varepsilon + \frac{\varepsilon^2}{2}\right)^{p/2} \\
        & = 2^p \sum_{n=0}^\infty \binom{p/2}{n} \left(\varepsilon + \frac{\varepsilon^2}{2}\right)^n, \quad \text{and} \\
        \|{\v^p}\|\|{\w_\varepsilon^p}\| & = \sqrt{2}\sqrt{1 + (1 + \varepsilon)^{2p}} \\
        & = \sqrt{2}\sqrt{1 + \sum_{n=0}^\infty \binom{2p}{n} \varepsilon^n} \\
        & = 2\sqrt{1 + \frac{1}{2}\sum_{n=1}^\infty \binom{2p}{n} \varepsilon^n} \\
        & = 2\sum_{m=0}^\infty \binom{1/2}{m}\left(\frac{1}{2}\sum_{n=1}^\infty \binom{2p}{n} \varepsilon^n\right)^m.
    \end{align*}
    Expanding the first few terms of these sums explicitly, the terms involving $\varepsilon^0$ and $\varepsilon^1$ cancel out, and we have
    \begin{align*}
        \|\v\|^p\|\w_\varepsilon\|^p - \ip{\v}{\w_\varepsilon}^p & = 2^p \left(\sum_{n=0}^\infty\left[\binom{p/2}{n} \left(\varepsilon + \frac{\varepsilon^2}{2}\right)^n - \binom{p}{n} \left(\frac{\varepsilon}{2}\right)^n\right]\right) \\
        & =\frac{p 2^p}{8}\varepsilon^2 + O(\varepsilon^3)
    \end{align*}
    and
    \begin{align*}
        \|\v^p\|\|\w_\varepsilon^p\| - \ip{\v^p}{\w_\varepsilon^p} & = 2\sum_{m=0}^\infty \binom{1/2}{m}\left(\frac{1}{2}\sum_{n=1}^\infty \binom{2p}{n} \varepsilon^n\right)^m - \left(1 + \sum_{n=0}^\infty \binom{p}{n} \varepsilon^n\right) \\
        & = \frac{p^2}{4}\varepsilon^2 + O(\varepsilon^3).
    \end{align*}
    It is straightforward to show that $p^2/4 > p2^p/8$ when $p \in (1,2)$; thus, we conclude that Inequality~\eqref{ineq:p12_counter} holds for sufficiently small $\varepsilon > 0$.
\end{example}

The above examples show that, despite Inequality~\eqref{ineq:integer_p} holding for all positive integers $p$, it cannot be extended to non-integers in the intervals $(0,1)$ or $(1,2)$. We conjecture that these two intervals comprise precisely the only positive real numbers for which the inequality does not hold:

\begin{conjecture}\label{conj:non_integer}
    Let $p \geq 2$ be a real number. For any pair of vectors $\v,\w \in \R_{>0}^n$, we have
    \begin{align}\label{ineq:nonint_conj}
        \| \v^p \|\| \w^p \| - \ip{\v^p}{\w^p} \leq \|\v\|^{p}\|\w\|^{p} - \ip{\v}{\w}^{p}.
    \end{align}
\end{conjecture}

We have numerically tested this conjecture extensively for $2 \leq n \leq 10$ and non-integer values of $p \in (2,10)$, and after more than $10^{10}$ randomly-generated vectors, we have found no counterexamples. This inequality, if true, is tight, since the left- and right-hand sides are equal (and both equal to $0$) when $\v = \w = \e_1$.

To illustrate the strength of this conjecture, consider the plots presented in Figure~\ref{fig:noninteger_conjecture}, which were created by $10^4$ randomly generated pairs of unit vectors and randomly generated values of $p \in [0,5]$. Each dot in the figure shows the value of $p$ on the $x$-axis and the value of the right-hand side of Inequality~\eqref{ineq:nonint_conj} minus its left-hand side on the $y$-axis. The conjecture states that if $p \in [2,\infty)$, then this quantity is non-negative, which clearly seems to be true. We see that if $p \in (0,1)$ or $p \in (1,2)$, then this claim is no longer true, since when $n = 2$, some dots in these regions are below the $x$-axis, as we noted in our earlier examples (like Example~\ref{exam:p12}). When $n = 3$, Figure~\ref{fig:noninteger_n3} perhaps looks like it does not have any points with $\text{RHS} - \text{LHS} < 0$, but they are all just close enough to the axis that we cannot quite see them. For instance, there is a point located at approximately $(p, \text{RHS} - \text{LHS}) = (1.2403, -0.0036)$. Furthermore, any $2$-dimensional example can be embedded in higher dimensions to give examples with $\text{RHS} - \text{LHS} < 0$ for $p \in (1,2)$ and any $n \geq 2$.

\input{fig2.tex}

We note that the top curve of these figures appears to be
\[
    \text{RHS} - \text{LHS} = 1 - \left(\left\lfloor \frac{n+1}{2} \right\rfloor \cdot \left\lceil \frac{n+1}{2} \right\rceil\right)^{(1-p)/2},
\]
which is attained by the normalizations of the vectors $\v = (1,\ldots,1,0,\ldots,0)$ and $\w = (0,\ldots,0,1,\ldots,1)$, where $\v$ has $\lfloor (n+1)/2\rfloor$ nonzero entries and $\w$ has $\lceil (n+1)/2\rceil$ nonzero entries, so there is a single entry (near the center of these vectors) that is nonzero in both $\v$ and $\w$.

\section*{Acknowledgements}
The authors thank Carlo Beenakker for providing the key insight that led to Figure~\ref{fig:geometric_proof} and the geometric discussion in Section~\ref{sec:intuitive_geometric} \cite{BeeMO}. N.J.\ was supported by NSERC Discovery Grant number RGPIN-2022-04098. S.P.\ was supported by NSERC Discovery Grant number 1174582, the Canada Foundation for Innovation (CFI) grant number 35711, and the Canada Research Chairs (CRC) Program grant number 231250.

\bibliographystyle{alpha}
\bibliography{ref}
\end{document}